\documentclass{amsart}

\usepackage{graphicx}

\newtheorem{theorem}{Theorem}[section]
\newtheorem{lemma}[theorem]{Lemma}
\newtheorem{proposition}[theorem]{Proposition}

\theoremstyle{definition}

\theoremstyle{remark}
\newtheorem{remark}[theorem]{Remark}

\numberwithin{equation}{section}

\usepackage{ulem}
\usepackage{amsrefs}

\usepackage{color}

\let\oldmarginpar\marginpar
\renewcommand\marginpar[1]{\-\oldmarginpar[\raggedleft\footnotesize #1]
{\raggedright\footnotesize #1}}

\usepackage[T1]{fontenc}
\usepackage{lmodern}
\usepackage[latin1]{inputenc}

\usepackage{setspace}
\usepackage{indentfirst}

\usepackage{multicol}
\usepackage{geometry}
\usepackage{hyperref}
\usepackage{amsmath,amssymb}
\usepackage{mathrsfs}
\usepackage{amsthm}

\theoremstyle{plain}

\newtheorem*{theorem*}{Theorem}
\newtheorem*{corollary*}{Corollary}

\newcommand{\cV}{\mathcal{V}}

\newcommand{\cE}{\mathcal{E}}

\newcommand{\C}{\mathbb{C}}
\newcommand{\CP}{\mathbb{C}\mathbf{P}}
\newcommand{\R}{\mathbb{R}}
\newcommand{\Z}{\mathbb{Z}}

\renewcommand{\emptyset}{\varnothing}
\newcommand{\Lie}{\text{Lie}}

\DeclareMathOperator{\Ima}{Im}
\DeclareMathOperator{\Hom}{Hom}

\begin{document}

\title{Poisson Cohomology of holomorphic toric Poisson manifolds. II.}

\author{Wei Hong}
\address{School of Mathematics and Statistics, Wuhan University, Wuhan, 430072, China}
\address{Hubei Key Laboratory of Computational Science, Wuhan University, Wuhan, 430072, China}
\email{hong\textunderscore  w@whu.edu.cn}


\begin{abstract}
In this paper, we give a description of holomorphic multi-vector fields on  smooth compact toric varieties, which generalizes Demazure's result of holomorphic vector fields on toric varieties. Based on the result, we compute the Poisson cohomology groups of holomorphic toric Poisson manifolds, \i.e., toric varieties endowed with 
$T$-invariant holomorphic Poisson structures. 
\end{abstract}

\maketitle


\section{introduction}\label{sect-Intr}

This paper is a sequel of \cite{Hong 19}. It contains two main parts.
The first part of this paper is devoted to the study of holomorphic multi-vector fields on toric varieties.
In \cite{Demazure}, Demazure described all the holomorphic vector fields on smooth compact toric varieties. In this paper, we give a description of all holomorphic multi-vector fields on smooth compact toric varieties, which generalizes Demazure's results of holomorphic vector fields on toric varieties.

Recall that a toric variety \cite{Cox} is an irreducible variety $X$ such that
\begin{enumerate}
\item $T=(\C^*)^n$ is a Zariski open set of $X$, 
\item the action of $T=(\C^*)^n$ on itself extends to an action of $T=(\C^*)^n$ on $X$.
\end{enumerate}

Let $N=\Hom(\C^*,T)\cong\Z^n$ and $M=\Hom(T, \C^*)$.
Then $M\cong \Hom_\Z(N,\Z)$ and $N\cong \Hom_\Z(M,\Z)$.
Toric varieties can be described by the lattice $N\cong\Z^n$ and a fan $\Delta$ in $N_{\mathbb{R}}=N\otimes_\Z\mathbb{R}\cong\mathbb{R}^n$. A toric variety associated with $(\Delta, N)$ is denoted by  $X_\Delta$. One may consult \cite{Cox}, \cite{Fulton} and \cite{Oda} for more details of toric varieties. 

Let $X=X_ \Delta$ be a smooth compact toric variety associated with a fan $ \Delta$ in $N_\R$. The $T$-action on $X$ induces a $T$-action on 
$H^0(X, \wedge^k T_{X})$, the vector space of holomorphic $k$-vector fields on $X$. 
Denote by $V_I^k$ the weight space corresponding to the character $I\in M$. 
We have
\begin{equation}\label{VIk-eqn}
H^0(X, \wedge^kT_{X})=\bigoplus_{I\in M}V_I^k.
\end{equation}
Since $H^0(X, \wedge^kT_{X})$ is a finite dimensional vector space, 
there are only finite elements $I\in M$ such that $V_I^k\neq 0$.

The next theorem describes the holomorphic multi-vector fields on toric varieties.  The vector space $V_I^k(\Delta)$ and the set $S_k(\Delta), P_{\Delta}( n-i)$ appearing in Theorem \nameref{General-multiVect-thm} are defined in Section \ref{sect-multi}.

\begin{theorem*}[A]\label{General-multiVect-thm}
Let $X=X_ \Delta$ be a smooth compact toric variety associated with a fan $ \Delta$ in $N_\R$. 
\begin{enumerate}
\item  We have
\begin{equation}\label{General-multiVect-thm-eqn1}
H^0(X,\wedge^kT_{X})=\bigoplus_{I\in S_k(\Delta)}V_I^k(\Delta)
\end{equation}
for all $0\leq k\leq n$, where $V_I^k(\Delta)=V_{-I}^k$ for all $I\in S_k(\Delta)$.
\item We have
\begin{equation}\label{General-multiVect-thm-eqn2}
\dim H^0(X,\wedge^kT_{X})
=\sum_{0\leq i\leq k}~\sum_{F_j\in P_{\Delta}( n-i)}{{n-i}\choose{k-i}}\cdot\#(int(F_j)\cap M)
\end{equation}
for all $0\leq k\leq n$.
\end{enumerate}
\end{theorem*}

We give some comments for Theorem \nameref{General-multiVect-thm}:
\begin{itemize}
\item In the case of $k=1$, the above theorem was proved by Demazure \cite{Demazure}. 
The set $$R(N, \Delta)=S_1(\Delta)\backslash\{0\}$$ is called the root system for 
$(N, \Delta)$.
\item In the case of $k=n$, we have the well-known results
\begin{gather*}
H^0(X,\wedge^nT_{X})\cong
\bigoplus_{I\in P_{\Delta}\cap M}\C\chi^I\cdot   W^n,\\
\dim H^0(X,\wedge^nT_{X})=\#(P_{\Delta}\cap M).
\end{gather*}
\item The special case of $X=\CP^n$ was proved in \cite{Hong 19}. 
We would like to point out a sign mistake in the Theorem 3.3 of \cite{Hong 19}, where $V_I^k$ should be the weight space corresponding to the character $-I$. 
\item 
If $X$ is a toric Fano manifold, Equation \eqref{General-multiVect-thm-eqn2} can be obtained by Theorem $3.6$ in \cite{Materov 02}  
since $\wedge^k T_X\cong\Omega^{n-k}_X\otimes\wedge^n T_X$.
\end{itemize}

The second part of this paper is devoted to the study of holomorphic Poisson manifolds, especially, the computation of Poisson cohomology groups of $T$-invariant holomorphic Poisson structures on toric varieties.

Recall that a holomorphic Poisson manifold is a complex manifold $X$ equipped with a holomorphic bivector field $\pi$ such that $[\pi,\pi]=0$, where $[\cdot,\cdot]$ is the Schouten bracket. Holomorphic Poisson manifolds are studied by many mathematicians from different viewpoints.
The algebraic geometry of Poisson manifolds was first studied by Bondal \cite{Bondal} and Polishchuk \cite{Polishchuk}.  Deformation quantization of Poisson varieties was studied by Kontsevich \cite{Kontsevich 01}.
Hitchin \cite{Hitchin 06, Hitchin 11, Hitchin 12} and Gualtieri \cite{Gualtieri 11} investigated holomorphic Poisson manifolds as a special case of generalized complex manifolds. 
The relation of holomorphic Poisson manifolds and Lie algebroids were revealed in
 \cite{L-S-X 08}. 
And Poisson structures on flag varieties were studied in \cite{B-G-Y 06, G-Y 09}.

The Poisson cohomology groups $H^\bullet_\pi (X) $ of a holomorphic Poisson manifold  $(X,\pi)$ is the cohomology group of the complex of sheaves:
\begin{equation}\label{PoissonCoh-eqn}
\mathcal{O}_{X}\xrightarrow{d_{\pi}}T_{X}\xrightarrow{d_{\pi}}.....
\xrightarrow{d_{\pi}}\wedge^{i-1}T_{X}\xrightarrow{d_{\pi}}\wedge^{i}T_{X}
\xrightarrow{d_{\pi}}\wedge^{i+1}T_{X}\xrightarrow{d_{\pi}}......
\xrightarrow{d_{\pi}}\wedge^{n}T_{X},
\end{equation}
where $d_{\pi}=[\pi,\cdot]$ and $\dim X=n$.
The Poisson cohomology groups of holomorphic Poisson manifolds are computed in various situations \cite{Hong-Xu 11, Mayansky 15, C-F-P 16, Hong 19}.

In this paper, we compute the Poisson cohomology groups for holomorphic toric Poisson manifolds. 
Recall that a holomorphic toric Poisson manifold  \cite{Hong 19} is a smooth toric variety $X$, endowed with a $T$-invariant holomorphic Poisson structure $\pi$ on $X$, 
and $\pi$ is called a holomorphic toric Poisson structure on $X$.
The name "toric Poisson structures" comes from \cite{Caine}, where
$T$-invariant $(1,1)$-type Poisson structures on toric varieties are studied.
We also notice that $T$-invariant holomorphic Poisson structures on the products of flag varieties are explored in \cite{Lu-Mouquin 15}.

Based on the results in Theorem \nameref{General-multiVect-thm}, we get
the next theorem, which gives the Poisson cohomology groups of holomorpic toric Poisson manifolds. The sets $S_k^{\pi}(\Delta)$ and $S^\pi(\Delta, i)$ appearing in Theorem \nameref{General-cohomology-thm} are defined in Section \ref{sect-Pocoh}. 

\begin{theorem*}[B]\label{General-cohomology-thm}
Let $X=X_\Delta$ be a smooth compact toric variety of dimension $n$ and
let $\pi$ be a holomorphic toric Poisson structure on $X$.
Assume that 
$H^i(X,\wedge^jT_{X})=0$ for all $i>0$ and $0\leq j\leq n$.

\begin{enumerate}
\item 
For $0\leq k\leq n$, we have
\begin{equation*}
H_{\pi}^k(X)\cong\bigoplus_{I\in S_k^{\pi}(\Delta)}V_I^{k}(\Delta).
\end{equation*}
Moreover, 
\begin{equation*}
\dim H_{\pi}^k(X)=\sum_{i=0}^{k}{{n-i}\choose{k-i}}\# S^\pi(\Delta, i).
\end{equation*}
\item  
We have $H_{\pi}^k(X)=0$ for $k>n$.
\end{enumerate}
\end{theorem*}

Theorem \nameref{General-cohomology-thm} generalizes the results in \cite{Hong 19}, 
where we compute the Poisson cohomology groups for toric Poisson structures on $\CP^n$ and $\C^n$.

As an application of Theorem  \nameref{General-cohomology-thm}, we have
\begin{corollary*}[C]\label{Fano-Pcohomology}
Let $X=X_{\Delta}$ be a toric Fano manifolds and $\pi$ be a holomorphic toric Poisson structure on $X$. Then the Poisson cohomology groups $H_{\pi}^\bullet(X)$ 
is given by Theorem  \nameref{General-cohomology-thm}.
\end{corollary*}

There are many interesting questions for further study. For example, we can study the Gerstenhaber algebra structures of holomorphic multi-vector fields on toric varieties,
and the Gerstenhaber algebra structures of the Poisson cohomology groups of holomorphic toric Poisson manifolds.
The modular class of holomorphic Poisson manifolds were studied in \cite{Brylinski-Zuckerman 99, Dulgushev 09}. A natural question is: when is a holomorphic toric Poisson manifold unimodular? However, those will be left for future works.

The paper is organized in the following way.
In Section \ref{sect-Pre}, we recall some necessary results on toric varieties and holomorphic Poisson manifolds. 
In Section \ref{sect-multi}, we first introduce some terminologies and study their properties, then we prove Theorem \nameref{General-multiVect-thm}. 
In Section \ref{sect-Pocoh}, we prove Theorem \nameref{General-cohomology-thm} and Corollary \nameref{Fano-Pcohomology}.

{\bf Acknowledgements} 
We would like to thank Yu Qiao, Xiang Tang and Ping Xu for helpful discussions and comments.

\section{Preliminary}\label{sect-Pre}
\subsection{Toric varieties}
Let $X=X_\Delta$ be a smooth compact toric variety associated with a fan $\Delta$ in $N_\R\cong\R^n$. Recall that $X_\Delta$ is smooth if and only if each cone 
$\sigma\subset\Delta$ is smooth, \i.e., $\sigma$ is generated by a subset of a basis of $N$. And $X_{\Delta}$ is compact if and only if $$|\Delta|=\bigcup_{\sigma\in\Delta}\sigma=N_{\R}.$$
Denote by $\Delta(k)$ $(0\leq k\leq n)$ the set of all $k$-dimensional cones in $\Delta$.
Let $U_{\sigma}\subseteq X_{\Delta}$ be the affine variety associated with a cone 
$\sigma\in\Delta(n)$. Then we have $X=\bigcup_{\sigma\in\Delta(n)} U_{\sigma}$.
As $X$ is smooth, we have $U_\sigma\cong\C^n$ for all $\sigma\in\Delta(n)$.
Moreover, we have $T\subset U_\sigma\subset X$ for all $\sigma\in\Delta(n)$.

Each element $m$ in $M$ gives rise to a character $\chi^m\in Hom(T,\C^*)$,
which can also be considered as a rational function on $X$.

Let $N_\C=N\otimes_{\Z}\C$. 
Since $T\cong N\otimes_\Z\C^*$, we have $\Lie(T)\cong N_\C$, 
where $\Lie(T)$ denotes the Lie algebra of $T$. 
We define a map $\rho:N_\C\rightarrow \mathfrak{X}(X)$ by
\begin{equation} \label{rho-def-eqn}
\rho:N_\C=N\otimes_\Z\C\cong \Lie(T)\rightarrow \mathfrak{X}(X),
\end{equation}
where $\Lie(T)\rightarrow \mathfrak{X}(X)$ is defined by the infinitesimal action of the Lie algebra $\Lie(T)$ on $X$.
Since the action map $T\times X\rightarrow X$ is holomorphic, 
the images of $\rho$ are holomorphic vector fields on $X$.
By abuse of notation, the induced maps
\begin{equation} \label{rho-def-eqn2}
\wedge^k N_\C\rightarrow\mathfrak{X}^k(X)
\end{equation}
 are also denoted by $\rho$ for $2\leq k\leq n$.
 
 Let 
 \begin{equation}\label{rho-def-eqn3}
 W=\rho(N_\C),\quad
W^k=\wedge^k  W\quad\text{and}\quad W^0=\C. 
\end{equation}
Then $W^k$ is a subspace of $H^0(X, \wedge^kT_{X})$. 
And the map $\wedge^k N_\C\xrightarrow{\rho} W^k$ is an isomorphism.

\begin{lemma}\label{VIk-lem}
Let $X$ be a smooth toric variety.
Let $V_I^k$ be the weight space corresponding to the character $I\in M$ for the 
$T$-action on $H^0(X,\wedge^k T_X)$.
\begin{enumerate}
\item We have $V_0^k=W^k$, where $V_0^k$ is the vector space consisting of all $T$-invariant holomorphic $k$-vector fields on $X$.
\item
For any holomorphic $k$-vector field $v\in V_{-I}^k$, there exists a unique 
 $k$-vector field $w\in W^k$, such that
\begin{equation}\label{VIKlem-eqn}
v|_T=\chi^I\cdot w|_T,
\end{equation}
where $v|_T$ and $\chi^I\cdot w|_T$ are the restrictions of $v$ and $\chi^I\cdot w$ on $T\subseteq X$.
\end{enumerate}
\end{lemma}
\begin{proof}
\begin{enumerate}
\item
By Equation \eqref{rho-def-eqn3}, we have
\begin{equation}\label{VIk-lem-eqn1}
W\cong \Lie(T)\cong N_\C\quad\text{and}\quad W^k\cong\wedge^k \Lie(T)\cong\wedge^k N_\C,
\end{equation}
where $\wedge^k \Lie(T)$ can be considered as the vector space of the $T$-invariant vector fields on $T$.
 
For any holomorphic $T$-invariant $k$-vector field $v\in V_0^k$, the restriction of $v$ 
on $T\subseteq X$ is also $T$-invariant. 
Thus there exists a holomorphic $k$-vector field $w\in W^k$, such that  
\begin{equation}\label{VIk-lem-eqn2}
v|_{T}=w|_{T}.
\end{equation}
Since $v$ and $w$ are holomorphic on $X$, and $T$ is dense in $X$, 
by Equation \eqref{VIk-lem-eqn2}, 
we have $$v=w$$ on $X$. 
Hence we have $V_0^k\subseteq W$.

On the other hand, any $w\in W^k$ is a $T$-invariant holomorphic $k$-vector field 
on $X$. Hence we have $W^k\subseteq V_0^k$.

By the arguments above, we have $V_0^k=W^k$ .

\item 
For any holomorphic $k$-vector field $v\in V_{-I}^k$, we have 
\begin{equation}\label{VIklem-eqn1}
 t_{*}(v)=\chi^{-I}(t)\cdot v
\end{equation}
for all $t\in T$, where $t_{*}(v)$ is the induced action of $t\in T$ 
on $v\in H^0(X,\wedge^k T_X)$, 
and $\chi^{-I}(t)$ is the value of $\chi^{-I}$ at $t\in T$. 

We will prove that $\chi^{-I}\cdot v|_T$ is a $T$-invariant holomorphic $k$-vector field 
on $T$. 
For any $t\in T$ and $p\in T\subseteq X$, we have
\begin{equation}\label{VIklem-eqn2}
(t_{*}(\chi^{-I}\cdot v))(p)=\chi^{-I}(t^{-1}\cdot p)\cdot (t_{*}(v))(p).
\end{equation}
By Equation \eqref{VIklem-eqn1}, we have
\begin{equation*}
( t_{*}(v))(p)=\chi^{-I}(t)\cdot v(p).
\end{equation*}
Since 
\begin{equation*}
\chi^{-I}(t^{-1}\cdot p)=\chi^{-I}(t^{-1})\chi^{-I}(p)=\chi^{I}(t)\chi^{-I}(p)
\end{equation*}
by Equation \eqref{VIklem-eqn2}, we have
$$t_{*}(\chi^{-I}\cdot v)(p)=\chi^{-I}(p)\cdot v(p)=(\chi^{-I}\cdot v)(p)$$
for all $t\in T$ and $p\in T\subseteq X$.
Therefore $\chi^{-I}\cdot v|_T$ is a $T$-invariant holomorphic $k$-vector field on $T$.
As a consequence, there exists $w\in W^k$ such that 
$$w|_T=\chi^{-I}\cdot v|_T,$$
or equivalently, $$v|_T=\chi^I\cdot w|_T.$$
\end{enumerate}
\end{proof}

By Lemma \ref{VIk-lem}, any holomorphic $k$-vector fields $v\in V_{-I}^k$ can be written as $v=\chi^I\cdot w$ on $T\subseteq X$, where $I\in M$ and $w\in W^k$.
 In general, $v=\chi^I\cdot w$ is a meromorphic $k$-vector field on $X$. 
 By abuse of notations, if $v=\chi^I\cdot w$ has moveable singularity on $X$, 
 we also use $v=\chi^I\cdot w$ to represent the corresponding holomorphic $k$-vector field on $X$,  and we say that $v=\chi^I\cdot w$ is holomorphic on $X$ in this paper.
 
 We have the following lemma.
\begin{lemma}\label{VIk-lem1}
Let $X$ be a smooth toric variety.  Let $v=\chi^I\cdot w$ be a $k$-vector field on $X$, where $I\in M$ and $w\in W^k$. Then $v=\chi^I\cdot w\in V_{-I}^k$ 
if and only if $v=\chi^I\cdot w$ is holomorphic on $X$.
\end{lemma}
\begin{proof}
\begin{enumerate}
\item
$``\Longleftarrow"$ \quad For any $w\in W^k$ and $v=\chi^I\cdot w$, 
suppose that $v$ is holomorphic on $X$. We will show that $v\in V_{-I}^k$.

By the similar proof as in Lemma \ref{VIk-lem}, we have
\begin{equation}\label{VIklem-eqn3}
t_*(v)|_T=t_*(\chi^I\cdot w)|_T=\chi^{-I}(t)(\chi^I\cdot w)|_T=\chi^{-I}(t)\cdot v|_T
\end{equation}
for all $t\in T$.
If $v=\chi^I\cdot w$ is holomorphic on $X$, 
then $t_*(v)$ and $\chi^{-I}(t)\cdot v$ are both holomorphic on $X$.
Since $T$ is dense in $X$,  the Equation \eqref{VIklem-eqn3} implies that
\begin{equation}
t_*(v)=\chi^{-I}(t)\cdot v
\end{equation}
for all $t\in T$. Hence we have $v\in V_{-I}^k$.

\item
$``\Longrightarrow"$ \quad If 
$v=\chi^I\cdot w\in V_{-I}^k\subseteq H^0(X,\wedge^k T_X)$, 
then $v$ is necessarily holomorphic on $X$.
\end{enumerate}
\end{proof}

\subsection{Poisson cohomology groups and holomorphic toric Poisson manifolds}
Let $(X,\pi)$ be a holomorphic Poisson manifolds.  The Poisson cohomology groups of 
$(X,\pi)$ are defined in Equation \eqref{PoissonCoh-eqn}. 
In general, the Poisson cohomology groups are difficult to compute.
However, the following lemmas give a way to compute the Poisson cohomology groups
of some holomorphic Poisson manifolds. 

\begin{lemma}\cite{L-S-X 08}
  The Poisson cohomology of a holomorphic
  Poisson manifold $(X,\pi)$ is isomorphic to the  total cohomology of the double complex \\
$$\begin{array}{ccccccc}
......& &......& &......& & \\
d_{\pi}\big\uparrow & & d_{\pi}\big\uparrow & & d_{\pi}\big\uparrow & &  \\
\Omega^{0,0}(X,T^{2,0}X) & \xrightarrow{\bar{\partial}} &
 \Omega^{0,1}(X, T^{2,0}X) & \xrightarrow{\bar{\partial}} &
 \Omega^{0,2}(X, T^{2,0}X) &\xrightarrow{\bar{\partial}}   &
 ......\\
 d_{\pi}\big\uparrow & & d_{\pi}\big\uparrow & & d_{\pi}\big\uparrow & &  \\
\Omega^{0,0}(X,T^{1,0}X) & \xrightarrow{\bar{\partial}} &
 \Omega^{0,1}(X, T^{1,0}X) & \xrightarrow{\bar{\partial}} &
 \Omega^{0,2}(X,T^{1,0}X) &\xrightarrow{\bar{\partial}}   &
 ......\\
 d_{\pi}\big\uparrow & & d_{\pi}\big\uparrow & & d_{\pi}\big\uparrow & &  \\
\Omega^{0,0}(X, T^{0,0}X) & \xrightarrow{\bar{\partial}} &
\Omega^{0,1}(X, T^{0,0}X) & \xrightarrow{\bar{\partial}} & \Omega^{0,2}(X, T^{0,0}X) & \xrightarrow{\bar{\partial}} &
 ......\\
\end{array}$$
\end{lemma}

By spectral sequence, we get the following lemma.
\begin{lemma}\label{LSX-lem}\cite{Hong-Xu 11}
Let $(X,\pi)$ be a holomorphic Poisson manifold. If all the higher
cohomology groups $H^{i}(X,\wedge^{j}T_{X})$ vanish for $i>0$, then the
Poisson cohomology  $H^\bullet_\pi (X)$
is isomorphic to the cohomology of the complex
\begin{equation}
0\rightarrow H^{0}(X,\mathcal{O}_{X})\xrightarrow {d_{\pi}}
H^{0}(X,T_{X})\xrightarrow{d_{\pi}}
H^{0}(X,\wedge^{2}T_{X})\xrightarrow{d_{\pi}}
\ldots \xrightarrow{d_{\pi}}H^{0}(X,\wedge^{n}T_{X})\rightarrow 0,
\end{equation}
where $d_{\pi}=[\pi,\cdot]$.
\end{lemma}  

Let $X=X_\Delta$ be a smooth toric variety associated with a fan $\Delta$ in $N_\R$.
It is natural to study the $T$-invariant holomorphic Poisson structures, \i.e., 
the holomorphic toric Poisson structures on $X$. 
The following proposition gives a description of all holomorphic toric Poisson structures 
on $X$.

\begin{proposition}\cite{Hong 19}\label{ToricPoisson-prop}
Let $X=X_\Delta$ be a smooth toric variety associated with a fan $\Delta$ in $N_\R$. 
Suppose that $\{e_1, e_2,\ldots, e_n\}$ is a basis of $N\subset N_\C$. 
Let $v_i=\rho(e_i)$ $(1\leq i\leq n)$  be holomorphic vector fields on $X$, 
where $\rho:N_\C\rightarrow \mathfrak{X}(X)$ is defined in  Equation \eqref{rho-def-eqn}.
Then $\pi$ is a holomorphic toric Poisson structure 
on $X$ if and only if $\pi$ can be written as
\begin{equation} \label{TP-prop-eqn}
\pi=\sum_{1\leq i<j\leq n}a_{ij} v_i\wedge v_j ,
\end{equation}
where $a_{ij}$ $(1\leq i<j\leq n)$ are complex constants.
\end{proposition}
 
By Proposition \ref{ToricPoisson-prop}, 
for any holomorphic toric Poisson structure $\pi$,
there exists a unique element $\Pi\in\wedge^2 N_\C$ such that 
$$\rho(\Pi)=\pi.$$ 
If the holomorphic toric Poisson structure $\pi$ is given 
in Equation \eqref{TP-prop-eqn}, then we have
\begin{equation}\label{Pi-eqn}
\Pi=\sum_{1\leq i<j\leq n}a_{ij} e_i\wedge e_j .
\end{equation}

\section{Holomorphic multi-vector fileds on toric varieties}\label{sect-multi}

\subsection{Toric varieties, lattices and polytopes}\label{polytope-sect}

\subsubsection{The polytope $P_{\Delta}$ and the set $S_\Delta$} 
Let $X=X_\Delta$ be a smooth compact toric variety.
Suppose that $\alpha_t ~(1\leq t\leq r)$ are all one dimensional cones in $\Delta(1)$. 
Let  $e(\alpha_t)\in N$ $(1\leq t\leq r)$ be the corresponding primitive elements, i.e., the unique generator of $\alpha_t \cap N$. 
Let
\begin{equation*}
E(\Delta)=\{e(\alpha_1),\ldots, e(\alpha_r)\}.
\end{equation*}

Let $P_\Delta$ be the polytope in $M_\R$ defined by
\begin{equation}\label{PDelta-eqn}
P_\Delta=\bigcap_{\alpha_t\in \Delta(1)}\{I\in M_\R\mid \langle I,e(\alpha_t)\rangle\geq -1\}.
\end{equation}
Since $X_{\Delta}$ is compact,  we have $|\Delta|=N_{\R}$,
thus $P_{\Delta}$ is a compact polytope in $M_{\R}$. Let 
\begin{equation*}
S_\Delta=\{I\mid I\in M\cap P_{\Delta}\}.
\end{equation*}
Then $S_{\Delta}$ is a non-empty finite set and $0\in S_{\Delta}$.

We denote by $P_\Delta(i)$ the set of all $i$-dimensional faces of the polytope $P_\Delta$.  Let
\begin{equation}\label{SDeltai}
S(\Delta,i)=\bigcup_{F_j\in P_\Delta(n-i)}\{I\in int(F_j)\cap M\}
\end{equation}
 for $0\leq i\leq n$, where $int(F_j)$ denotes the interior of $F_j$.
 Then we have 
 \begin{equation*}
 S(\Delta,i)\cap S(\Delta,j)=\emptyset
 \end{equation*}
 for all $0\leq i\neq j\leq n$.
 Let
\begin{equation*}
 S_k(\Delta)=\bigcup_{0\leq i\leq k}S(\Delta,i)
 \end{equation*}
  for all $0\leq k\leq n$. 
 Then we have
$$S_0(\Delta)\subseteq S_1(\Delta)\subseteq\ldots\subseteq S_n(\Delta)=S_{\Delta}$$
and $$S_{k+1}(\Delta)=S_k(\Delta)\cup S(\Delta,k+1).$$

\begin{lemma}\label{S0Delta-lem}
Let $X=X_\Delta$ be a smooth compact toric variety of dimension $n$. Then we have
\begin{equation*}
S_0(\Delta)=\{0\}.
\end{equation*}
\end{lemma}
\begin{proof}
For any $I\in S_0(\Delta)$, we have
\begin{equation}\label{S0Delta-lem-eqn1}
\langle I, e(\alpha_t)\rangle\geq 0
\end{equation}
for all $\alpha_t\in \Delta(1)$. 
Since $X$ is compact, we have
\begin{equation*}
\bigcup_{\sigma\in\Delta(n)}\sigma=N_\R.
\end{equation*} 
Therefore for any $x\in N_\R$, there exists $\sigma\in\Delta(n)$, such that $x\in\sigma$.
Since $X$ is smooth, $x\in\sigma$ can be written as
\begin{equation}\label{S0Delta-lem-eqn2}
x=\sum_{\alpha_t\in\Delta(1)\cap\sigma} \lambda_t e(\alpha_t), \quad\lambda_t\geq 0.
\end{equation}
 By Equation  \eqref{S0Delta-lem-eqn1} and Equation \eqref{S0Delta-lem-eqn2}, 
 we have
 \begin{equation*}
\langle I,x\rangle\geq 0
\end{equation*}
for all $x\in N_\R$, which implies $I=0$. 
Hence $S_0(\Delta)=\{0\}.$
\end{proof}

For any $I\in S_{\Delta}$, there exists a unique face $F_I(\Delta)$ of $P_{\Delta}$, 
such that $I\in int(F_I(\Delta))\cap M$. 
Let
\begin{equation*}
E_I(\Delta)=\{e(\alpha_t)\in E(\Delta)\mid \langle I,e(\alpha_t)\rangle=-1\}.
\end{equation*} 
Suppose that $E_I(\Delta)=\{e(\alpha_{s_1}),\ldots e(\alpha_{s_l})\}$, 
where $1\leq s_1<\ldots<s_l\leq r$.
Let $F^{\perp}_I(\Delta)$ be a subspace of $N_\R$ defined by
\begin{equation*}
F^{\perp}_I(\Delta)=\sum_{1\leq t\leq l} \R\cdot e(\alpha_{s_t}).
\end{equation*}
We define
\begin{equation*} 
F^{\perp}_I(\Delta)=0 \quad\text{if}\quad E_I(\Delta)=\emptyset.
\end{equation*}
Then $F^{\perp}_I(\Delta)$ can be considered as the normal space of $F_I(\Delta)$.

\begin{lemma}\label{SDelta-lem}
\begin{enumerate}
\item
For any $I\in S_{\Delta}$, the following statements are equivalent:
\begin{enumerate}
\item $I\in S(\Delta,i)$,
\item $F_I(\Delta)\in P_{\Delta}(n-i)$,
\item $\dim F^{\perp}_I(\Delta)=i$.
\end{enumerate}
\item
For any $I\in S_{\Delta}$, the following statements are equivalent:
\begin{enumerate}
\item $I\in S_k(\Delta)$,
\item $\dim F^{\perp}_I(\Delta)\leq k$.
\end{enumerate}
\end{enumerate}
\end{lemma}

\subsubsection{The polyhedron $P_{\sigma}$ and the set $S_\sigma$} 
Let $\sigma$ be a smooth cone of dimension $n$. Then $\sigma$ can be written as
$$\sigma=\sum_{t=1}^{n}\R_{\geq 0}\cdot e_t(\sigma),$$
where $\{e_1(\sigma),\ldots, e_n(\sigma)\}$ is a $\Z$-basis of $N$.
Let
\begin{equation*}
 E(\sigma)=\{e_1(\sigma),\ldots, e_n(\sigma)\}.
 \end{equation*}

Let $P_{\sigma}$ be the polyhedron defined by
\begin{equation}\label{Psigma-eqn}
P_\sigma=\{I\in M_\R\mid \langle I, e_t(\sigma)\rangle\geq -1~ \text{for all}~ 1\leq t\leq n\}.
\end{equation}
Let
\begin{equation*}
S_{\sigma}=P_{\sigma}\cap M.
\end{equation*}

We denote by $P_\sigma(i)$ the set of all $i$-dimensional faces of the polyhedron $P_\Delta$.  Let
\begin{equation*}
S(\sigma,i)=\bigcup_{F_j\in P_\sigma(n-i)}\{I\in int(F_j)\cap M\}
\end{equation*}
 for $0\leq i\leq n$, where $int(F_j)$ denotes the interior of $F_j$.
 Then we have 
 \begin{equation*}
 S(\sigma,i)\cap S(\sigma,j)=\emptyset
 \end{equation*}
 for all $0\leq i\neq j\leq n$. 
Let
\begin{equation*}
 S_k(\sigma)=\bigcup_{0\leq i\leq k}S(\sigma,i)
 \end{equation*}
  for all $0\leq k\leq n$. 
 Then we have
$$S_0(\sigma)\subseteq S_1(\sigma)\subseteq\ldots\subseteq S_n(\sigma)
=S_{\sigma}.$$

For any $I\in S_{\sigma}$,  there exists a unique face $F_I(\sigma)$ of $P_{\sigma}$, 
such that $I\in int(F_I(\sigma))\cap M$. 
 Let 
$$E_I(\sigma)=\{e_t(\sigma)\in E(\sigma)\mid \langle I,e_t(\sigma)\rangle=-1\}.$$

Suppose that $E_I(\sigma)=\{e_{s_1}(\sigma),\ldots e_{s_j}(\sigma)\}$, 
where $1\leq s_1<\ldots<s_j\leq n$.
Let $F^{\perp}_I(\sigma)$ be a subspace of $N_\R$ defined by 
\begin{equation*}
F^{\perp}_I(\sigma)=\sum_{1\leq t\leq j} \R\cdot e_{s_t}(\sigma).
\end{equation*}
And we define
\begin{equation*} 
F^{\perp}_I(\sigma)=0 \quad\text{if}\quad E_I(\sigma)=\emptyset.
\end{equation*}

Then $F^{\perp}_I(\sigma)$ can be considered as the normal space of $F_I(\sigma)$.

Similar to Lemma \ref{SDelta-lem}, we have
\begin{lemma}\label{Ssigma-lem}
\begin{enumerate}
\item
For any $I\in S_{\sigma}$, the following statements are equivalent:
\begin{enumerate}
\item $I\in S(\sigma,j)$,
\item $F_I(\sigma)\in P_{\sigma}(n-j)$,
\item $\dim F^{\perp}_I(\sigma)=j$,
\item $\# E_I(\sigma)=j$.
\end{enumerate}
\item
For any $I\in S_{\sigma}$, the following statements are equivalent:
\begin{enumerate}
\item $I\in S_k(\sigma)$,
\item $\dim F^{\perp}_I(\sigma)\leq k$,
\item $\# E_I(\sigma)\leq k$.
\end{enumerate}
\end{enumerate}
\end{lemma}

Now we have the following proposition.
\begin{proposition}\label{PFS-lem}
Let $X_\Delta$ be a smooth compact toric variety. 
\begin{enumerate}
\item
We have
\begin{equation*}
P_{\Delta}=\bigcap_{\sigma\in\Delta(n)} P_{\sigma}
\end{equation*}
and 
\begin{equation*}
S_{\Delta}=\bigcap_{\sigma\in\Delta(n)} S_{\sigma}.
\end{equation*}
\item
For any $\sigma\in\Delta(n)$ and $I\in P_{\Delta}\subset P_{\sigma}$, 
$F^{\perp}_I(\sigma)$ is a subspace of $F^{\perp}(\Delta)$. 
Moreover, we have
\begin{equation*}
F^{\perp}_I(\Delta)=\sum_{\sigma\in\Delta(n)}F^{\perp}_I(\sigma).
\end{equation*}
\item
We have
\begin{equation*}
S_k(\Delta)\subseteq\bigcap_{\sigma\in\Delta(n)}S_k(\sigma).
\end{equation*}
\end{enumerate}
\end{proposition}

\begin{proof}
\begin{enumerate}
\item
Since $X_\Delta$ is a smooth compact toric variety, we have 
$|\Delta|=\bigcup_{\sigma\in\Delta(n)}|\sigma|$. 
As a consequence, we have 
\begin{equation*}
\bigcup_{\sigma\in\Delta(n)}\{e_1(\sigma),\ldots,e_n(\sigma)\}=\{e(\alpha_1),\ldots,e(\alpha_r)\}.
\end{equation*}
By Equation \eqref{PDelta-eqn} and Equation \eqref{Psigma-eqn}, 
we get that $$P_{\Delta}=\bigcap_{\sigma\in\Delta(n)} P_{\sigma}.$$
Since $S_\Delta=P_\Delta\cap M$ and $S_\sigma=P_\sigma\cap M$, 
we get that
$$S_{\Delta}=\bigcap_{\sigma\in\Delta(n)} S_{\sigma}.$$

\item
Since $E_I(\sigma)\subseteq E_I(\Delta)$,  
we get that $F^{\perp}_I(\sigma)$ is a subspace of $F^{\perp}_I(\Delta)$.  
And by $$\bigcup_{\sigma\in\Delta(n)} E_I(\sigma)=E_I(\Delta),$$
we get that $F^{\perp}_I(\Delta)=\sum_{\sigma\in\Delta(n)}F^{\perp}_I(\sigma)$. 
\item
For any $I\in S_k(\Delta)$, by Lemma \ref{SDelta-lem}, we have
$$\dim F^{\perp}_I(\Delta)\leq k.$$
For any $\sigma\in\Delta(n)$, 
since $F^{\perp}_I(\sigma)$ is a subspace of $F^{\perp}_I(\Delta)$, we have 
 $$\dim F^{\perp}_I(\sigma)\leq\dim F^{\perp}_I(\Delta)\leq k.$$
 By Lemma \ref{Ssigma-lem}, we have $I\in S_k(\sigma)$.
 
 Therefore, we have
 $$S_k(\Delta)\subseteq S_k(\sigma)$$
 for all $\sigma\in\Delta(n)$. And consequently,
 $$S_k(\Delta)\subseteq\bigcap_{\sigma\in\Delta(n)}S_k(\sigma).$$
\end{enumerate}
\end{proof}

\subsection{The vector space $N_{I}^k(\Delta)$ and $N_{I}^k(\sigma)$}
\subsubsection{The vector space $N_{I}^k(\Delta)$}\label{VIk-Delta-sect}
Let $X_\Delta$ be a smooth compact toric variety of dimension $n$.
For any  $I\in S_{\Delta}$, since $S_{\Delta}=\bigcup_{0\leq i\leq n}S(\Delta,i)$,
there exists a unique integer $0\leq i\leq n$, such that $I\in S(\Delta,i)$. 
We denote by $|I_\Delta|$ the integer $i$, \i.e., $|I_\Delta|=i$. Then we have
\begin{equation}\label{|I|-eqn}
 |I_\Delta|=i\Longleftrightarrow I\in S(\Delta,i).
\end{equation}

Suppose that $I\in int(F_I(\Delta))\cap M$, by Lemma \ref{SDelta-lem},
we have $F_I(\Delta)\in P_\Delta(n-i)$. 
And suppose that $$E_I(\Delta)=\{e(\alpha_{s_1}),\ldots e(\alpha_{s_l})\},$$ 
where $1\leq s_1<\ldots<s_l\leq r$.
By Lemma \ref{SDelta-lem}, 
$F^{\perp}_I(\Delta)=\sum_{1\leq t\leq l} \R\cdot e(\alpha_{s_t})$
 is a $i$-dimensional subspace of $N_\R$.
No loss of generality, suppose that 
$\{e(\alpha_{s_1}),\ldots,e(\alpha_{s_{i}})\}\subseteq\{e(\alpha_{s_1}),\ldots,e(\alpha_{s_{l}})\}$ 
is a basis of $F^{\perp}_I(\Delta)$, where $i\leq l$.

Let
 \begin{gather*}
 \cE_I(\Delta)=e(\alpha_{s_1}) \wedge\ldots\wedge e(\alpha_{s_i})\in\wedge^{|I_\Delta|} N.
\end{gather*}
Since $\{e(\alpha_{s_1}),\ldots,e(\alpha_{s_i})\}$ is a basis of $F^{\perp}_I(\Delta)$, 
we have
\begin{equation*} 
\wedge^i F^{\perp}_I(\Delta)=\R\cdot\cE_I(\Delta).
\end{equation*}
Thus  $\cE_I(\Delta)$ is well defined up to a scalar if we  choose different  basis 
$\{e(\alpha_{s_1}),\ldots,e(\alpha_{s_i})\}\subseteq E_I(\Delta)$ for $F^{\perp}_I(\Delta)$.

Let $N_I^k(\Delta)$ be the subspace of $\wedge^k N_{\C}$ ($0\leq k\leq n$) defined by
\begin{equation*}
N_I^k(\Delta)=
\begin{cases}
\C\cdot\mathcal{E}_I(\Delta)\wedge  (\wedge^{k-|I_\Delta|} N_\C)\quad 
&\text{for}~ |I_\Delta|\leq k,\\
0\quad &\text{for}~|I_\Delta|>k.
\end{cases}
\end{equation*}

The following lemma gives a description of the vector space $N_I^k(\Delta)$.
\begin{lemma}\label{WDelta-lem}
Let $X_{\Delta}$ be a smooth compact toric variety.
For any $I\in S_\Delta$ and $x\in\wedge^k N_\C$, 
the following statements are equivalent:
\begin{enumerate}
\item 
$x\in N_I^k(\Delta)$,
\item $x\wedge e(\alpha_t)=0$ for all $e(\alpha_t)\in E_I(\Delta)$,
\item $x\wedge y=0$ for all $y\in F^{\perp}_I(\Delta)$.
\end{enumerate}
\end{lemma}
\begin{proof}
\begin{enumerate}
\item [(a)]
For $I\in S_\Delta$, in the case of $|I_\Delta|\leq k$, we suppose that $0\leq |I_\Delta|=i\leq k$ and 
 \begin{equation*}
 \cE_I(\Delta)=e(\alpha_{s_1}) \wedge\ldots\wedge e(\alpha_{s_i})\in\wedge^{i} N,
\end{equation*}
where $\{e(\alpha_{s_1}),\ldots,e(\alpha_{s_i})\}$ is a basis of 
$F^{\perp}_I(\Delta)\subseteq N_\C$.

We extend $\{e(\alpha_{s_1}),\ldots,e(\alpha_{s_i})\}$ to be a basis to $N_\C$,
\i.e.,  suppose that $$\{e(\alpha_{s_1}),\ldots,e(\alpha_{s_i}), f_1,\ldots, f_{n-i}\}$$
is a basis of $N_\C$. 
By some computations under the basis, we can prove Lemma \ref{WDelta-lem}. 
Here we omit the detail.
\item [(b)]
For $I\in S_\Delta$, In the case of $|I_\Delta|>k$, we have $N_I^k(\Delta)=0$. 
Lemma \ref{WDelta-lem} can be proved similarly.
\end{enumerate}
\end{proof}

\subsubsection{The vector space $N_{I}^k(\sigma)$}
Let $\sigma=\sum_{t=1}^{n}\R_{\geq 0}e_t(\sigma)$ be a smooth cone 
 of dimension $n$, where $\{e_1(\sigma),\ldots, e_n(\sigma)\}$ is a $\Z$-basis of $N$.
 Let $U_{\sigma}\cong\C^n$ be the affine toric variety associated the cone $\sigma$.
 Then we have $U_{\sigma}\cong\C^n$.
  
For any  $I\in S_{\sigma}$, since $S_{\sigma}=\bigcup_{0\leq j\leq n}S(\sigma,j)$,
there exists a unique integer $0\leq j\leq n$, such that $I\in S(\sigma,j)$. 
We denote by $|I_\sigma|$ the integer $j$, \i.e., $|I_\sigma|=j$.

Suppose that 
$$E_I(\sigma)=\{e_{s_1}(\sigma),\ldots, e_{s_j}(\sigma)\},$$
where $0\leq s_1<\ldots<s_j\leq n$.

Let
 \begin{equation*}
 \cE_I(\sigma)=e_{s_1}(\sigma) \wedge\ldots\wedge e_{s_j}(\sigma)\in\wedge^{|I_\sigma|} N.
\end{equation*}

Let $N_I^k(\sigma)$ be the subspace of $\wedge^k N_\C$ defined by
\begin{equation*}
N_I^k(\sigma)=
\begin{cases}
\C\cdot\mathcal{E}_I(\sigma)\wedge(\wedge^{k-|I_\sigma|} N_\C)
\quad\text{for}~ |I_\sigma|\leq k,\\
0\quad\text{for}~|I_\sigma|>k,
\end{cases}
\end{equation*}
for all $0\leq k\leq n$.

Similar to Lemma \ref{WDelta-lem}, we have
\begin{lemma}\label{Wsigma-lem}
Let $\sigma$ be a smooth cone of dimension $n$ in $N_\R$.
For any $I\in S_\sigma$ and $x\in\wedge^k N_\C$, 
the following statements are equivalent:
\begin{enumerate}
\item
$x\in N_I^k(\sigma)$,
\item
$x\wedge e_t(\sigma)=0$ for all $e_t(\sigma)\in E_I(\sigma)$,
\item 
$x\wedge y=0$ for all $y\in F^{\perp}_I(\sigma)$.
\end{enumerate}
\end{lemma}

\begin{remark}
Lemma \ref{Wsigma-lem}  is true in the case of $|I_\sigma|>k$, where $N_I^k(\sigma)=0$.
\end{remark}

\begin{proposition} \label{NIk-lem}
Let $X_{\Delta}$ be a smooth compact toric variety associated with a fan $\Delta$ in $N_{\R}$.
For all $I\in S_\Delta=\bigcap_{\sigma\in\Delta(n)}S_\sigma$,
we have
\begin{equation}\label{capNIk-eqn}
\bigcap_{\sigma\in\Delta(n)}N_I^k(\sigma)=N_I^k(\Delta).
\end{equation} 
Especially, we have
\begin{equation*}
\bigcap_{\sigma\in\Delta(n)}N_I^k(\sigma)=0
\quad\text{for all}~ I\in S_\Delta\backslash S_k(\Delta).
\end{equation*}
\end{proposition}

\begin{proof}
Since $|\Delta|=\bigcup_{\sigma\in\Delta(n)}\sigma$, we have 
\begin{equation*}
E_I(\Delta)=\bigcup_{\sigma\in\Delta(n)} E_I(\sigma).
\end{equation*}
By Lemma \ref{Wsigma-lem}, $x\in N_I^k(\sigma)$ if and only if 
$x\wedge e(\alpha_t)=0$ for all $e(\alpha_t)\in E_I(\sigma)$.
 As a consequence, 
$$x\in\bigcap_{\sigma\in\Delta(n)}N_I^k(\sigma)$$
 if and only if
$$x\wedge e(\alpha_t)=0$$
 for all $e(\alpha_t)\in \bigcup_{\sigma\in\Delta(n)}E_I(\sigma)=E_I(\Delta)$.
 By Lemma \ref{WDelta-lem},  
 $$x\wedge e(\alpha_t)=0$$ for all $e(\alpha_t)\in E_I(\Delta)$ if and only if 
 $$x\in N_I^k(\Delta).$$
 Therefore, we have
 \begin{equation}\label{NIk-lem-eqn1}
 \bigcap_{\sigma\in\Delta(n)}N_I^k(\sigma)=N_I^k(\Delta).
 \end{equation}
 For any $I\in S_\Delta\backslash S_k(\Delta)$, we have $|I_\Delta|>k$, 
 which implies $N_I^k(\Delta)=0$. By Equation \eqref{NIk-lem-eqn1},
 we have $$\bigcap_{\sigma\in\Delta(n)}N_I^k(\sigma)=0$$
  for all $I\in S_\Delta\backslash S_k(\Delta)$.
\end{proof}

\begin{remark}
By Proposition \ref{NIk-lem},
$N_I^k(\Delta)$ is a subspace of $N_I^k(\sigma)$ for any $\sigma\in\Delta(n)$.
\end{remark}

\subsection{The vector spaces $V_I^k(\Delta)$ and $V_I^k(\sigma)$}
\subsubsection{The vector space $V_I^k(\sigma)$}
Let $\sigma$ be a smooth cone of dimension $n$ in $N_\R$.
Let $U_{\sigma}\cong\C^n$ be the affine variety associated with $\sigma$. 

Let 
\begin{equation*}
W(\sigma)=\{w|_{U_\sigma}\mid w\in W\},
\end{equation*}
where $w|_{U_\sigma}$ is the restriction of the holomorphic vector field $w$ 
on $U_\sigma$.
Let
\begin{equation*}
W^k(\sigma)=\{w|_{U_\sigma}\mid w\in W^k\}.
\end{equation*}
Then $W^k(\sigma)$ is a subspace of $H^0(U_\sigma,\wedge^kT_{U_\sigma})$.
And we have $W^k(\sigma)=\wedge^k W(\sigma)$ $(0\leq k\leq n)$.

For any $I\in S_\sigma$, let
\begin{equation*}
\mathcal{V}_I(\sigma)=\rho(\cE_I(\sigma))|_{U_\sigma}\in W^{|I_\sigma|}(\sigma).
\end{equation*}
And let
\begin{equation}\label{WIkSigma-eqn}
W_I^k(\sigma)=\rho(N_I^k(\sigma))|_{U_\sigma}=
\begin{cases}
\C\cdot\mathcal{V}_I(\sigma)\wedge  W^{k-|I_\sigma|}(\sigma)
\quad&\text{for}~ |I_\sigma|\leq k,\\
0\quad &\text{for}~|I_\sigma|>k,
\end{cases}
\end{equation}
for all $0\leq k\leq n$. 
Then $W_I^k(\sigma)$ is a subspace of $W^k(\sigma)$.

If $\sigma$ is a smooth cone of dimension $n$ in $N_\R$, then $E(\sigma)$ is a $\Z$-basis of $N$, and we have $U_{\sigma}\cong\C^n$. 
By Lemma $5.1$ in \cite{Hong 19}, we have
\begin{lemma}\label{VIksigma-lem1}
Let $\sigma$ be a smooth cone of dimension $n$ in $N_\R$.
Let $v=\chi^I\cdot w$ be a $k$-vector field on $U_\sigma$,
where $I\in M$ and $w\in W^k(\sigma)~ (w\neq 0)$. 
Then $v=\chi^I\cdot w$ is holomorphic on $U_\sigma$ 
if and only if 
\begin{equation*}
I\in S_k(\sigma)\quad\text{and}\quad w\in W_I^k(\sigma).
\end{equation*}
\end{lemma}

Let
\begin{equation*}
V_I^k(\sigma)=\chi^I\cdot W_I^k(\sigma)=\{\chi^I\cdot w\mid w\in W^k(\sigma)\}.
\end{equation*}
By Lemma \ref{VIksigma-lem1}, for any $I\in S_k(\sigma)$,
elements in $V_I^k(\sigma)$ are holomorphic $k$-vector fields on $X$.
Hence $V_I^k(\sigma)$ can be considered a subspace of $H^0(U_\sigma,\wedge^kT_{U_\sigma})$ for any $I\in S_k(\sigma)$.

For a smooth cone $\sigma$ of dimension $n$ in $N_{\R}$, 
the $T$-action on $U_\sigma$ induces a $T$-action on 
$H^0(U_\sigma,\wedge^kT_{U_\sigma})$.
We denote by $V_I^k(U_\sigma)$ the weight space corresponding to the character $I\in M$. 

The following lemma also comes from Lemma $5.1$ in \cite{Hong 19}.
\begin{lemma}\label{VIksigma-lem}
Let $\sigma$ be a smooth cone of dimension $n$ in $N_{\R}$. 
Let $U_\sigma\cong\C^n$ be the affine toric variety associated with the cone $\sigma$.
Then we have
\begin{equation*}
V_{-I}^k(U_\sigma)=
\begin{cases}
V_I^k(\sigma)\quad &\text{for all}\quad I\in S_k(\sigma),\\
0\quad &\text{for all}\quad I\in M\backslash S_k(\sigma).
\end{cases}
\end{equation*}
\end{lemma}

\subsubsection{The vector space $V_I^k(\Delta)$}
\begin{lemma}\label{weightspace-lem1}
Let $X=X_\Delta$ be a compact smooth toric variety.
Let $v=\chi^I\cdot\rho(x)$ be a $k$-vector field on $X$,
where $I\in M$ and $x\in\wedge^k N_\C ~(x\neq 0)$. 
Then $v$ is holomorphic on $X$ if and if 
\begin{equation*}
x\in N_I^k(\Delta) \quad \text{and}
\quad I\in\bigcap_{\sigma\in\Delta(n)}S_k(\sigma).
\end{equation*}
\end{lemma}
\begin{proof}
Since $X=\cup_{\sigma\in\Delta(n)}U_{\sigma}$, $v=\chi^I\cdot\rho(x)$ 
is holomorphic on $X$ if and only if $v=\chi^I\cdot\rho(x)$ is holomorphic on 
$U_\sigma$ for all $\sigma\in\Delta(n)$. 
By Lemma \ref{VIksigma-lem1},
$v=\chi^I\cdot\rho(x)$ is holomorphic on $U_\sigma$ if and only if
$I\in S_k(\sigma)$ and $\rho(x)\in W_I^k(\sigma)$. 
However, by Equation \eqref{WIkSigma-eqn}, we have
$$\rho(x)\in W_I^k(\sigma)\Longleftrightarrow x\in N_I^k(\sigma).$$
Thus $v$ is holomorphic on $X$ if and if 
\begin{equation*}
x\in \bigcap_{\sigma\in\Delta(n)}N_I^k(\sigma) \quad \text{and}
\quad I\in\bigcap_{\sigma\in\Delta(n)}S_k(\sigma).
\end{equation*}
\end{proof}

Let $X=X_\Delta$ be a smooth compact toric variety of dimension $n$. 
Let
\begin{equation*}
\mathcal{V}_I(\Delta)=\rho(\cE_I(\Delta))\in W^{|I_\Delta|}
\end{equation*}
for $I\in S_\Delta$.
Let $W_I^k(\Delta)$ ($0\leq k\leq n$) be defined by
\begin{equation*}
W_I^k(\Delta)=\rho(N_I^k(\Delta))=
\begin{cases}
\C\cdot\mathcal{V}_I(\Delta)\wedge  W^{k-|I_\Delta|}\quad &\text{for}~ |I_\Delta|\leq k,\\
0\quad &\text{for}~|I_\Delta|>k.
\end{cases}
\end{equation*}
Then $W_I^k(\Delta)$ is a subspace of $W^k$.
Let
\begin{equation*}
V_I^k(\Delta)=\chi^I\cdot W_I^k(\Delta)=\{\chi^I\cdot w\mid w\in W_I^k(\Delta)\}.
\end{equation*}
The elements in $V_I^k(\Delta)$ are considered as meromorphic $k$-vector fields on $X=X_\Delta$.

\begin{lemma}\label{weightspace-lem2}
Let $X=X_\Delta$ be a smooth compact toric variety of dimension $n$. For all
$I\in S_k(\Delta)$, any $k$-vector field in $V_I^k(\Delta)$ is holomorphic on $X$.
\end{lemma}
\begin{proof}
For any $k$-vector field $v=\chi^I\cdot w\in V_I^k(\Delta)$, 
there exists $x\in N_I^k(\Delta)$, 
such that $w=\rho(x)$. By Proposition \ref{NIk-lem}, $x\in N_I^k(\Delta)$ implies
$x\in \bigcap_{\sigma\in\Delta(n)}N_I^k(\sigma)$.
By Proposition \ref{PFS-lem}, $I\in S_k(\Delta)$ implies
$I\in\bigcap_{\sigma\in\Delta(n)}S_k(\sigma)$.
By Lemma \ref{weightspace-lem1}, $v$ is holomorphic on $X$.
\end{proof}

By Lemma \ref{weightspace-lem2}, $V_I^k(\Delta)$ can be considered as a subspace of 
$H^0(X,\wedge^k T_X)$. 

\begin{proposition}\label{weightspace-lem} 
Let $X=X_\Delta$ be a smooth compact toric variety of dimension $n$. 
Let $V_{I}^k$ be the weight space corresponding to the character $I\in M$ for 
the $T$-action on $H^0(X,\wedge^k T_X)$. 
\begin{enumerate}
\item
We have
\begin{equation*}
V_{-I}^k=
\begin{cases}
V_I^k(\Delta)\quad &\text{for all}\quad I\in S_k(\Delta),\\
0\quad &\text{for all}\quad I\in M\backslash S_k(\Delta).
\end{cases}
\end{equation*}
\item
For any $I\in S_k(\Delta)$, we have
\begin{equation*}
\dim V_I^k(\Delta)={{n-|I_\Delta|}\choose{k-|I_\Delta|}}.
\end{equation*}
\end{enumerate}
\end{proposition}

\begin{proof}
\begin{enumerate}
\item

\begin{enumerate}
\item 
We will prove here $V_{-I}^k=V_I^k(\Delta)$ for all $I\in S_k(\Delta)$.

By Lemma \ref{VIk-lem}, any element $v\in V_{-I}^k$ can be written as 
$$v=\chi^I\cdot w,$$ where $w\in W^k$.
And by Lemma \ref{VIk-lem1}, $v=\chi^I\cdot w\in V_{-I}^k$ if and only if 
$v=\chi^I\cdot w$ is holomorphic on $X$.

Suppose that $w=\rho(x)$, where $x\in\wedge^k N_\C$.
By Lemma \ref{weightspace-lem1},  $v$ is holomorphic on $X$ if and only if 
\begin{equation}\label{weightspace-lem-eqn1}
x\in \bigcap_{\sigma\in\Delta(n)}N_I^k(\sigma) \quad \text{and}
\quad I\in\bigcap_{\sigma\in\Delta(n)}S_k(\sigma).
\end{equation}

However, by Proposition \ref{NIk-lem}, we have
$$\bigcap_{\sigma\in\Delta(n)}N_I^k(\sigma)=N_I^k(\Delta)$$
for all $I\in S_k(\Delta)$. 
And by Proposition \ref{PFS-lem}, $I\in S_k(\Delta)$ implies 
$$I\in\bigcap_{\sigma\in\Delta(n)}S_k(\sigma).$$

As a consequence, for $I\in S_k(\Delta)$, 
$v=\chi^I\cdot\rho(x)$ is holomorphic on $X$ if and only if 
$$x\in N_I^k(\Delta),$$
which is equivalent to 
$$v\in V_I^k(\Delta).$$

Hence we have $V_{-I}^k=V_I^k(\Delta)$ for all $I\in S_k(\Delta)$.

\item
We will prove $V_{-I}^k=0$ for all $I\notin S_k(\Delta)$.

By Proposition \ref{PFS-lem}, we have
\begin{equation}\label{weightspace-lem-eqn2.2}
S_k(\Delta)\subseteq\bigcap_{\sigma\in\Delta(n)}S_k(\sigma).
\end{equation}
As a consequence, for $I\notin S_k(\Delta)$, we have
\begin{equation*} 
I\notin\bigcap_{\sigma\in\Delta(n)}S_k(\sigma)\quad\text{or}\quad
I\in\bigcap_{\sigma\in\Delta(n)}S_k(\sigma)\backslash S_k(\Delta).
\end{equation*}
By Lemma \ref{VIk-lem}, any $v\in V_{-I}^k$ can be written as $v=\chi^I\cdot w$, 
where $w\in W^k$.
By Lemma \ref{VIk-lem1}, $v=\chi^I\cdot w\in V_{-I}^k$ if and only if 
$v=\chi^I\cdot w$ is holomorphic on $X$.
Suppose that $w=\rho(x)$, where $x\in\wedge^k N_\C$. Then $v=\chi^I\cdot\rho(x)$.
By Lemma \ref{weightspace-lem1}, 
$v=\chi^I\cdot\rho(x)$ $(v\neq 0)$ is holomorphic on $X$ if and only if 
 \begin{equation}\label{weightspace-lem-eqn2}
x\in \bigcap_{\sigma\in\Delta(n)}N_I^k(\sigma)~(x\neq0)~\quad \text{and}
\quad I\in\bigcap_{\sigma\in\Delta(n)}S_k(\sigma).
\end{equation}

\begin{enumerate}
\item 
For any $I\notin\bigcap_{\sigma\in\Delta(n)}S_k(\sigma)$, 
by Equation \eqref{weightspace-lem-eqn2}, we have
\begin{equation}\label{weightspace-lem-eqn2.1}
V_{-I}^k=0.
\end{equation}

\item
For any 
$I\in(\bigcap_{\sigma\in\Delta(n)}S_k(\sigma))\backslash S_k(\Delta)$, 
by Proposition \ref{PFS-lem}, we have
\begin{equation}\label{weightspace-lem-eqn3}
I\in(\bigcap_{\sigma\in\Delta(n)}S_k(\sigma))\backslash S_k(\Delta)\subseteq (\bigcap_{\sigma\in\Delta(n)}S_\sigma)\backslash S_k(\Delta)=S_\Delta\backslash S_k(\Delta).
\end{equation}
By Proposition \ref{NIk-lem} and Equation \eqref{weightspace-lem-eqn3},  we have
\begin{equation}\label{weightspace-lem-eqn4}
\bigcap_{\sigma\in\Delta(n)}N_I^k(\sigma)=0
\end{equation}
for all $I\in(\bigcap_{\sigma\in\Delta(n)}S_k(\sigma))\backslash S_k(\Delta)$.
By Equation \eqref{weightspace-lem-eqn2} and Equation \eqref{weightspace-lem-eqn4}, 
we have 
\begin{equation}\label{weightspace-lem-eqn5}
V_{-I}^k=0\quad\text{for all}\quad
I\in(\bigcap_{\sigma\in\Delta(n)}S_k(\sigma))\backslash S_k(\Delta).
\end{equation}
\end{enumerate}

By Equation \eqref{weightspace-lem-eqn2.2}, Equation \eqref{weightspace-lem-eqn2.1} and Equation \eqref{weightspace-lem-eqn5}, 
we have $V_{-I}^k=0$ for $I\notin S_k(\Delta)$.

\end{enumerate}

\item
For any  $I\in S_k(\Delta)$, we have $0\leq |I_\Delta|\leq k$ . 
Suppose that $0\leq |I_\Delta|=i\leq k$, and suppose that
 \begin{gather*}
 \cE_I(\Delta)=e(\alpha_{s_1}) \wedge\ldots\wedge e(\alpha_{s_i})\in\wedge^{i} N,
\end{gather*}
where $\{e(\alpha_{s_1}),\ldots,e(\alpha_{s_i})\}$ is a basis of 
$F^{\perp}_I(\Delta)\subseteq N_\C$.

We extend $\{e(\alpha_{s_1}),\ldots,e(\alpha_{s_i})\}$ to be a basis to $N_\C$.
 Suppose that
 $$\{e(\alpha_{s_1}),\ldots,e(\alpha_{s_i}), f_1,\ldots, f_{n-i}\}$$
is a basis of $N_\C$. Then 
$$\{\rho(e(\alpha_{s_1})),\ldots,\rho(e(\alpha_{s_i})),\rho( f_1),\ldots, \rho(f_{n-i})\}$$
is a basis of $W$.
As a consequence,
$$\{\chi^I\cdot\rho(\cE_I(\Delta))\wedge\rho(f_{t_1})\wedge\ldots\rho(f_{t_{k-i}})\mid
1\leq t_1< \ldots<t_{k-i}\leq n-i \}$$
is a basis of the vector space 
$V_I^k(\Delta)=\chi^I\cdot\rho(\cE_I(\Delta))\wedge W^{k-i}$.
 Hence we have
$$\dim V_I^k(\Delta)={{n-i}\choose{k-i}}={{n-|I_\Delta|}\choose{k-|I_\Delta|}}.$$
\end{enumerate}
\end{proof}

\subsection{The proof of  Theorem \nameref{General-multiVect-thm}}
\begin{proof}
\begin{enumerate}
\item
By Equation \eqref{VIk-eqn}, we have
\begin{equation}\label{General-multiVect-thmpf-eqn1}
H^0(X, \wedge^kT_{X})=\bigoplus_{I\in M}V_I^k=\bigoplus_{I\in M}V_{-I}^k
\end{equation}
By Proposition \ref{weightspace-lem}, we have
\begin{equation}\label{General-multiVect-thmpf-eqn2}
V_{-I}^k=
\begin{cases}
V_I^k(\Delta)\quad &\text{for}~ I\in S_k(\Delta),\\
0\quad &\text{for}~I\notin S_k(\Delta).
\end{cases}
\end{equation}
By Equation \eqref{General-multiVect-thmpf-eqn1} and 
Equation \eqref{General-multiVect-thmpf-eqn2}, we have
\begin{equation}\label{H0k-thmpf-eqn}
H^0(X, \wedge^kT_{X})=\bigoplus_{I\in S_k(\Delta)}V_I^k(\Delta)
\end{equation} 
\item
For any $I\in S(\Delta, i)$, $|I_\Delta|=i$, by Proposition \ref{weightspace-lem}, we have
\begin{equation}\label{dimVIk-thmpf-eqn}
\dim V_I^k(\Delta)={{n-i}\choose{k-i}}.
\end{equation}
By Equation \eqref{SDeltai}, we have
\begin{equation}\label{countSD-eqn}
\#S(\Delta,i)=\sum_{F_j\in P_\Delta(n-i)}\# (int(F_j)\cap M).
\end{equation}
By Equation \eqref{dimVIk-thmpf-eqn} and Equation \eqref{countSD-eqn}, we have
\begin{equation}\label{dimVIKSD-eqn}
\sum_{I\in S(\Delta,i)}\dim V_I^k(\Delta)=
\sum_{F_j\in P_{\Delta}( n-i)}{{n-i}\choose{k-i}}\cdot\#(int(F_j)\cap M)
\end{equation}
Since $S_k(\Delta)=\bigcup_{i=0}^k S(\Delta, i),$ by Equation \eqref{H0k-thmpf-eqn}
and Equation \eqref{dimVIKSD-eqn}, we have
\begin{equation*}
\dim H^0(X, \wedge^kT_{X})=\sum_{i=0}^k \sum_{I\in S(\Delta,i)}\dim V_I^k(\Delta)
=\sum_{i=0}^k \sum_{F_j\in P_{\Delta}( n-i)}{{n-i}\choose{k-i}}\cdot\#(int(F_j)\cap M).
\end{equation*}
\end{enumerate}
\end{proof}

\section{Poisson cohomology of holomorphic toric Poisson manifolds}\label{sect-Pocoh}
\subsection{Holomorphic toric Poisson manifolds and the set $S_{\Delta}^{\pi}$} 
Let $X=X_\Delta$ be a compact smooth toric variety. 
Let $\pi$ be a holomorphic toric Poisson structure on $X$, 
and let $\Pi\in\wedge^2 N_\C$ be defined as in Equation \eqref{Pi-eqn}.

We consider the equation
\begin{equation}\label{General-IPi-eqn}
(\imath_{I}\Pi)\wedge \cE_I(\Delta)=0
\end{equation}
for all $I\in S_{\Delta}$. 
For $I\in S_{\Delta}$, Equation \eqref{General-IPi-eqn} is equivalent to 
\begin{equation}\label{General-IPi-eqn1}
\imath_{I}\Pi\in F^{\perp}_I(\Delta)\otimes\C.
\end{equation}

Let
\begin{gather}
S^{\pi}_\Delta=\{I\in S_{\Delta}\mid (\imath_{I}\Pi)\wedge \cE_I(\Delta)=0\},\\
S_k^{\pi}(\Delta)=\{I\in S_k(\Delta)\mid (\imath_{I}\Pi)\wedge \cE_I(\Delta)=0\},\\
S^{\pi}(\Delta,k)=\{I\in S(\Delta,k)\mid (\imath_{I}\Pi)\wedge \cE_I(\Delta)=0\}.
\end{gather}

Then we have
\begin{equation*}
 S^{\pi}_k(\Delta)=\bigcup_{0\leq i\leq k}S^{\pi}(\Delta,i)
 \end{equation*}
  for all $0\leq k\leq n$.  
  Moreover, we have
$$S^{\pi}_0(\Delta)\subseteq S^{\pi}_1(\Delta)\subseteq\ldots\subseteq S^{\pi}_n(\Delta)
=S^{\pi}_{\Delta}.$$

By Lemme \ref{VIk-lem}, we have $S_0(\Delta)=\{0\}$, which implies 
\begin{equation*}
S_0^\pi(\Delta)=\{0\}.
\end{equation*}

For any $I\in S(\Delta,n)$, since $\cE_I(\Delta)\in\wedge^n N_\C$, the equation
$$(\imath_{I}\Pi)\wedge \cE_I(\Delta)=0$$
holds automatically. Thus we have 
\begin{equation*}
S^\pi(\Delta,n)=S(\Delta,n).
\end{equation*}

\subsection{Lemmas for Theorem  \nameref{General-cohomology-thm}}
We give some lemmas here to prove Theorem  \nameref{General-cohomology-thm}.
\begin{lemma}\label{Poisson-Coh-lem1}
Let $X=X_\Delta$ be a smooth compact toric variety of dimension $n$ and
let $\pi$ be a holomorphic toric Poisson structure on $X$.
For $I\in S_\Delta$, we have
\begin{equation*}\label{General-PoCoh-lemma-eqn3}
[\pi,\chi^I\cdot \cV_I(\Delta)\wedge w]=\rho(\imath_{I}\Pi)\wedge(\chi^I\cdot \cV_I(\Delta)\wedge w)=\chi^I\cdot \rho((\imath_{I}\Pi)\wedge\cE_I(\Delta))\wedge w ,
\end{equation*}
where $|I_\Delta|\leq k\leq n$ and $w\in W^{k-|I_\Delta|}$.
\end{lemma}

The proof of Lemma \ref{Poisson-Coh-lem1} is the same as 
Lemma $4.4$ in \cite{Hong 19}. Here we skip it.

By Theorem \nameref{General-multiVect-thm}, we have
\begin{equation}\label{H0kSD-eqn1}
H^0(X,\wedge^kT_{X})=\bigoplus_{I\in S_k(\Delta)}V_I^k(\Delta).
\end{equation}
For any $I\in S_\Delta\backslash S_k(\Delta)$, since $|I_\Delta|>k$, we have
\begin{equation}\label{H0kSD-eqn2}
V_I^k(\Delta)=0\quad\text{for}\quad I\in S_\Delta\backslash S_k(\Delta).
\end{equation}

By Equation \eqref{H0kSD-eqn1} and Equation \eqref{H0kSD-eqn2},  we get
the following lemma.

\begin{lemma}\label{Poisson-Coh-lem2}
Let $X=X_\Delta$ be a smooth compact toric variety of dimension $n$. We have
\begin{equation}\label{H0kSD-eqn}
H^0(X,\wedge^kT_{X})=\bigoplus_{I\in S_\Delta}V_I^k(\Delta)
\end{equation}
for all $0\leq k\leq n$.
\end{lemma}

By Lemma \ref{Poisson-Coh-lem1} and Lemma \ref{Poisson-Coh-lem2}, we have the following lemma.
\begin{lemma}\label{Poisson-Coh-lem3}
Let $X=X_\Delta$ be a smooth compact toric variety of dimension $n$ and
let $\pi$ be a holomorphic toric Poisson structure on $X$.
\begin{enumerate}
\item
For any $I\in S_{\Delta}$ and $v\in V_I^k(\Delta)$ $(0\leq k\leq n)$, we have
\begin{equation}\label{Poisson-Coh-lem3-eqn0}
d_{\pi}v=\rho(\imath_{I}\Pi)\wedge v\in V_{I}^{k+1}(\Delta).
\end{equation}
Therefore we have $d_{\pi}(V_I^k(\Delta))\subseteq V_I^{k+1}(\Delta)$.
\item
The chain complex
\begin{equation*}
 0\rightarrow H^{0}(X,\mathcal{O}_{X})\xrightarrow {d_{\pi}}
H^{0}(X,T_{X})\xrightarrow{d_{\pi}}H^{0}(X,\wedge^{2}T_{X})
\xrightarrow{d_{\pi}}
\ldots \xrightarrow{d_{\pi}}H^{0}(X,\wedge^{n}T_{X})\rightarrow 0
\end{equation*}
splits to the direct sum of the sub-chain complex 
\begin{equation}\label{Poisson-Coh-lem3-eqn1}
0\xrightarrow{d_\pi} V_I^0(\Delta)\xrightarrow{d_\pi}V_I^1(\Delta)
\xrightarrow{d_\pi}V_I^2(\Delta)\xrightarrow{d_\pi}\ldots\xrightarrow{d_\pi}V_I^{n}(\Delta)\rightarrow 0
\end{equation}
for all $I\in S_\Delta$. 
The chain complex \eqref{Poisson-Coh-lem3-eqn1} can be written as
\begin{equation}\label{Poisson-Coh-lem3-eqn2}
0\rightarrow\ldots \rightarrow 0\rightarrow V_I^{|I_\Delta|}(\Delta)\xrightarrow{\rho(\imath_{I}\Pi)\wedge\cdot} V_I^{|I_\Delta|+1}(\Delta)\xrightarrow{\rho(\imath_{I}\Pi)\wedge\cdot} 
\ldots\xrightarrow{\rho(\imath_{I}\Pi)\wedge\cdot} V_I^n(\Delta)\rightarrow 0
\end{equation}
for all $I\in S_\Delta$.
\item 
The chain complex \eqref{Poisson-Coh-lem3-eqn2} is isomorphic to the chain complex
\begin{equation}\label{Poisson-Coh-lem3-eqn3}
0\rightarrow\ldots \rightarrow 0\rightarrow N_I^{|I_\Delta|}(\Delta)\xrightarrow{\imath_{I}\Pi\wedge\cdot} N_I^{|I_\Delta|+1}(\Delta)\xrightarrow{\imath_{I}\Pi \wedge\cdot} 
\ldots\xrightarrow{\imath_{I}\Pi\wedge\cdot} N_I^n(\Delta)\rightarrow 0
\end{equation}
for all $I\in S_\Delta$.
\end{enumerate}
\end{lemma}

\begin{proof}
\begin{enumerate}
\item
For any $I\in S_{\Delta}$, if $|I_\Delta|>k$, then we have $V_I^k(\Delta)=V_I^{k+1}(\Delta)=0$.
Thus in this case, the conclusion holds automatically.

For any $I\in S_{\Delta}$, if $|I_\Delta|\leq k$, then $v\in V_I^k(\Delta)$ can be written as
\begin{equation*}
v=\chi^I\cdot\cV_I(\Delta)\wedge w,
\end{equation*}
where $w\in W^{k-|I_\Delta|}$. 
By Lemma \ref{Poisson-Coh-lem1}, we have
\begin{equation*}
d_{\pi}v=\rho(\imath_{I}\Pi)\wedge v
=(-1)^{|I_\Delta|}\chi^I\cdot \cV_I(\Delta)\wedge(\rho(\imath_{I}\Pi)\wedge w)
\in V_{I}^{k+1}(\Delta). 
\end{equation*}
Thus we have $d_{\pi}(V_I^k(\Delta))\subseteq V_I^{k+1}(\Delta)$.
\item
Since $d_{\pi}(V_I^k(\Delta))\subseteq V_I^{k+1}(\Delta)$, 
by Lemma \ref{Poisson-Coh-lem2} and Equation \eqref{Poisson-Coh-lem3-eqn0},
 the chain complex
\begin{equation*}
 0\rightarrow H^{0}(X,\mathcal{O}_{X})\xrightarrow {d_{\pi}}
H^{0}(X,T_{X})\xrightarrow{d_{\pi}}H^{0}(X,\wedge^{2}T_{X})
\xrightarrow{d_{\pi}}
\ldots \xrightarrow{d_{\pi}}H^{0}(X,\wedge^{n}T_{X})\rightarrow 0
\end{equation*}
can split to the direct sum of the sub-chain complex 
\begin{equation}\label{Poisson-Coh-lem3-eqn1-1}
0\xrightarrow{d_\pi} V_I^0(\Delta)\xrightarrow{d_\pi}V_I^1(\Delta)
\xrightarrow{d_\pi}V_I^2(\Delta)\xrightarrow{d_\pi}\ldots\xrightarrow{d_\pi}V_I^{n}(\Delta)\rightarrow 0
\end{equation}
for all $I\in S_\Delta$. 
Since $$d_{\pi}v=\rho(\imath_{I}\Pi)\wedge v$$ for all $v\in V_I^k(\Delta)$, 
and since $$V_I^k(\Delta)=0$$ for all $k<|I_\Delta|$,
the chain complex \eqref{Poisson-Coh-lem3-eqn1-1} can be written as
\begin{equation*}
0\rightarrow\ldots \rightarrow 0\rightarrow V_I^{|I_\Delta|}(\Delta)\xrightarrow{\rho(\imath_{I}\Pi)\wedge\cdot} V_I^{|I_\Delta|+1}(\Delta)\xrightarrow{\rho(\imath_{I}\Pi)\wedge\cdot} 
\ldots\xrightarrow{\rho(\imath_{I}\Pi)\wedge\cdot} V_I^n(\Delta)\rightarrow 0
\end{equation*}
for all $I\in S_\Delta$. 
\item
Since $V_I^k(\Delta)=\chi^I\cdot\rho(N_I^k(\Delta)$, 
the chain complex \eqref{Poisson-Coh-lem3-eqn2} is isomorphic to the chian complex
\eqref{Poisson-Coh-lem3-eqn3} for all $I\in S_{\Delta}$.
\end{enumerate}
\end{proof}

The following lemma is a generalization of the Lemma 4.6 in \cite{Hong 19}. 
\begin{lemma}\label{Poisson-Coh-lem4}
Let $X=X_\Delta$ be a smooth compact toric variety of dimension $n$ and
let $\pi$ be a holomorphic toric Poisson structure on $X$.
\begin{enumerate}
\item
For any $I\in S_k^{\pi}(\Delta)$ $(0\leq k\leq n)$, \i.e., $I\in S_k(\Delta)$ satisfying
$(\imath_{I}\Pi)\wedge \cE_I=0$, we have
\begin{equation}\label{Poisson-Coh-lem4-eqn1}
d_\pi(\psi)=0
\end{equation}
for all $\psi\in V_I^k(\Delta)$.
\item
For any $I\in S_k(\Delta)\backslash S_k^{\pi}(\Delta)$ $(1\leq k\leq n)$,
\i.e., $I\in S_k(\Delta)$ satisfying $(\imath_{I}\Pi)\wedge \cE_I\neq0$,
 if $\psi\in V_I^k(\Delta)$ satisfies
 \begin{equation}\label{Poisson-Coh-lem4-eqn2}
d_\pi(\psi)=0, 
\end{equation}
 then there exists 
 $\phi\in V_I^{k-1}(\Delta)$ such that
 \begin{equation}\label{Poisson-Coh-lem4-eqn3}
 \psi=d_\pi(\phi).
 \end{equation}  
 \end{enumerate}
\end{lemma}

\begin{proof}
\begin{enumerate}
\item
For $I\in S_k^{\pi}(\Delta)\subseteq S_k(\Delta)$, we have $|I_\Delta|\leq k$.
 Any element $\psi\in V_I^k(\Delta)$ can be written as 
$$\psi=\chi^I\cdot\cV_I(\Delta)\wedge w=\chi^I\cdot\rho(\cE_I(\Delta))\wedge w,$$
where $w\in W^{k-|I_\Delta|}$. 
By Lemma \ref{Poisson-Coh-lem3}, we have
\begin{equation}\label{Poisson-Coh-lem4-eqn4} 
d_\pi(\psi)=\rho(\imath_{I}\Pi)\wedge\psi
=\chi^I\cdot\rho(\imath_{I}\Pi\wedge\cE_I(\Delta))\wedge w.
\end{equation}
Since $\imath_{I}\Pi\wedge\cE_I(\Delta)=0$, we have 
$$d_\pi(\psi)=0$$ 
for all $I\in S_k^{\pi}(\Delta)$ $(0\leq k\leq n)$ and $\psi\in V_I^k(\Delta)$.

\item
If $\psi=0\in V_I^k(\Delta)$, then we can choose $\phi=0\in V_I^{k-1}(\Delta)$.
Next we suppose that $\psi\in V_I^k(\Delta)$ satisfies $\psi\neq 0$. 

\begin{enumerate} 
\item
For $I\in S_k(\Delta)\backslash S_k^{\pi}(\Delta)$, we have $k\geq |I_\Delta|$ and 
$(\imath_{I}\Pi)\wedge \cE_I(\Delta)\neq 0$.

We first prove that for $\psi\in V_I^k(\Delta)$ $(\psi\neq0)$, 
the Equation \eqref{Poisson-Coh-lem4-eqn2} implies 
$$k>|I_\Delta|.$$
 By Lemma \ref{Poisson-Coh-lem1}, we have
 \begin{equation}\label{Poisson-Coh-lem4-eqn5}
 d_\pi(\chi^I \cdot\cV_I(\Delta))=\chi^I\rho(\imath_{I}\Pi\wedge\cE_I(\Delta)).
 \end{equation}
Since $V_I^{|I_\Delta|}=\C\chi^I \cdot\cV_I(\Delta)$, the condition $(\imath_{I}\Pi)\wedge \cE_I(\Delta)\neq 0$ and Equation \eqref{Poisson-Coh-lem4-eqn5}
 imply that for any $\psi\in V_I^{|I_\Delta|}(\Delta)$ $(\psi\neq 0)$, 
we have $d_\pi(\psi)\neq0$. 
Thus we get $k>|I_\Delta|$ if $d_\pi(\psi)=0$ for $\psi\in V_I^k(\Delta)$ $(\psi\neq 0)$.

\item
Any element $\psi\in V_I^k(\Delta)$ can be written as 
\begin{equation}\label{Poisson-Coh-lem4-eqn6}
\psi=\chi^I\cdot\rho(\cE_I(\Delta)\wedge x),
\end{equation}
where $x\in\wedge^{k-|I_\Delta|}N_\C$.
By Lemma \ref{Poisson-Coh-lem1}, we have
\begin{equation}\label{Poisson-Coh-lem4-eqn7}
 d_\pi(\psi)=\chi^I\cdot\rho(\imath_{I}\Pi\wedge\cE_I(\Delta)\wedge x).
 \end{equation}
 
Suppose that 
 \begin{equation*}
 \cE_I(\Delta)=e(\alpha_{s_1}) \wedge\ldots\wedge e(\alpha_{s_i})\in\wedge^{|I_\Delta|} N.
\end{equation*}
where $\{e(\alpha_{s_1}),\ldots,e(\alpha_{s_i})\}$ is a basis of $F^{\perp}_I(\Delta)$, 
and $|I_\Delta|=i$.

For any $I\in S_k(\Delta)\backslash S_k^{\pi}(\Delta)$, the condition 
$$(\imath_{I}\Pi)\wedge \cE_I(\Delta)\neq 0$$ 
implies that $\{\imath_{I}\Pi, ~e(\alpha_{s_1}),\ldots,e(\alpha_{s_i})\}$ 
 are $\C$-linearly independent vectors in $N_\C$.

By Equation \eqref{Poisson-Coh-lem4-eqn7}, $d_\pi(\psi)=0$ implies 
\begin{equation}\label{Poisson-Coh-lem4-eqn8}
 \imath_{I}\Pi\wedge\cE_I(\Delta)\wedge x
 =\imath_{I}\Pi\wedge e(\alpha_{s_1})\wedge\ldots\wedge e(\alpha_{s_i})\wedge x=0.
 \end{equation}
Since  $\{\imath_{I}\Pi, ~e(\alpha_{s_1}),\ldots,e(\alpha_{s_i})\}$ 
 are $\C$-linearly independent vectors in $N_\C$, 
 by Equation \eqref{Poisson-Coh-lem4-eqn8}
 there exist $x_0, x_1,\ldots, x_i\in \wedge^{k-|I_\Delta|-1}N_\C$, such that
 \begin{equation}\label{Poisson-Coh-lem4-eqn9}
 x=\imath_{I}\Pi\wedge x_0+e(\alpha_{s_1})\wedge x_1
 +\ldots+e(\alpha_{s_i})\wedge x_i.
 \end{equation}
 By Equation \eqref{Poisson-Coh-lem4-eqn6} and 
 Equation \eqref{Poisson-Coh-lem4-eqn9}, and with some computations, we have
\begin{equation}\label{Poisson-Coh-lem4-eqn10}
\psi=(-1)^i \rho(\imath_{I}\Pi)\wedge(\chi^I\cdot\rho(\cE_I(\Delta)\wedge x_0)).
\end{equation}
Let $\phi=(-1)^i \chi^I\cdot\rho(\cE_I(\Delta)\wedge x_0)\in V_I^{k-1}(\Delta)$. 
By Equation \eqref{Poisson-Coh-lem4-eqn10} and 
Lemma \ref{Poisson-Coh-lem1}, we have
 $$\psi=d_\pi(\phi).$$  
 \end{enumerate}
\end{enumerate}
\end{proof}

\subsection{The proof of Theorem  \nameref{General-cohomology-thm}}
\begin{proof}
By Lemma \ref{LSX-lem}, the Poisson cohomology  $H^\bullet_\pi (X)$
is isomorphic to the cohomology of the complex
\begin{equation}\label{General-cohomology-thm-eqn1}
0\rightarrow H^{0}(X,\mathcal{O}_{X})\xrightarrow {d_{\pi}}
H^{0}(X,T_{X})\xrightarrow{d_{\pi}}
H^{0}(X,\wedge^{2}T_{X})\xrightarrow{d_{\pi}}
\ldots \xrightarrow{d_{\pi}}H^{0}(X,\wedge^{n}T_{X})\rightarrow 0.
\end{equation}

\begin{enumerate}
\item
In the case of $k=0$, by Equation \eqref{General-cohomology-thm-eqn1},
we have $$H_\pi^0(X)\cong\C.$$
Since $S_0^{\pi}(\Delta)=S_0(\Delta)=\{0\}$, we have 
$$\bigoplus_{I\in S_0^\pi(\Delta)}V_I^0=V_0^0(\Delta)=\C.$$
Thus Theorem \nameref{General-cohomology-thm} is true in the case of $k=0$.

\item
In the case of $1\leq k\leq n$, by Lemma \ref{Poisson-Coh-lem3}, 
the chain complex \eqref{General-cohomology-thm-eqn1} can split to the direct sum of
the sub-chain complex 
\begin{equation*}
0\xrightarrow{d_\pi} V_I^0(\Delta)\xrightarrow{d_\pi}V_I^1(\Delta)
\xrightarrow{d_\pi}V_I^2(\Delta)\xrightarrow{d_\pi}\ldots\xrightarrow{d_\pi}V_I^{n}(\Delta)\rightarrow 0
\end{equation*}
for all $I\in S_\Delta$. 
As a consequence, we have
\begin{equation}\label{General-cohomology-thm-eqn2}
H_{\pi}^k(X)\cong\bigoplus_{I\in S_{\Delta}}
\frac{\ker(V_I^k(\Delta)\xrightarrow{d_{\pi}}V_I^{k+1}(\Delta))}
{\Ima(V_I^{k-1}(\Delta)\xrightarrow{d_{\pi}}V_I^{k}(\Delta))},
\end{equation}
where $V_I^{n+1}(\Delta)=0$.

\begin{enumerate}
\item
For any $I\in S_\Delta\backslash S_k(\Delta)$, we have $|I_\Delta|>k$, which implies 
$$ V_I^k(\Delta)=0$$ for all $I\in S_\Delta\backslash S_k(\Delta)$. 
Thus we have
\begin{equation}\label{General-cohomology-thm-eqn2.1}
\frac{\ker(V_I^k(\Delta)\xrightarrow{d_{\pi}}V_I^{k+1}(\Delta))}
{\Ima(V_I^{k-1}(\Delta)\xrightarrow{d_{\pi}}V_I^{k}(\Delta))} =0
\end{equation}
for all $I\in S_\Delta\backslash S_k(\Delta)$.
\item
For any $I\in S_k(\Delta)\backslash S_k^\pi(\Delta)$,
by Lemma \ref{Poisson-Coh-lem4}, we have 
\begin{equation*}
\ker( V_I^k(\Delta)\xrightarrow{d_{\pi}}V_I^{k+1}(\Delta))=
\Ima(V_I^{k-1}(\Delta)\xrightarrow{d_{\pi}}V_I^{k}(\Delta)).
\end{equation*}
Therefore
\begin{equation}\label{General-cohomology-thm-eqn3}
\frac{\ker(V_I^k(\Delta)\xrightarrow{d_{\pi}}V_I^{k+1}(\Delta))}
{\Ima(V_I^{k-1}(\Delta)\xrightarrow{d_{\pi}}V_I^{k}(\Delta))} =0
\end{equation}
for all $I\in S_k(\Delta)\backslash S_k^\pi(\Delta)$.

\item
For any $I\in S_k^\pi(\Delta)$, by Lemma \ref{Poisson-Coh-lem4},  
we have
\begin{equation}\label{General-cohomology-thm-eqn4}
\ker(V_I^k(\Delta)\xrightarrow{d_{\pi}}V_I^{k+1}(\Delta))=V_I^k(\Delta).
\end{equation}

\begin{enumerate}
\item 
For any $I\in S_{k-1}^\pi(\Delta)$, by Lemma \ref{Poisson-Coh-lem4},  we have
\begin{equation}\label{General-cohomology-thm-eqn4.1.1}
\Ima(V_I^{k-1}(\Delta)\xrightarrow{d_{\pi}}V_I^{k}(\Delta))=0.
\end{equation}
\item 
For any $I\in S_{k}^\pi(\Delta)\backslash S_{k-1}^\pi(\Delta)=S^\pi(\Delta,k)$, 
we have $|I_\Delta|=k>k-1$, which implies $V_I^{k-1}(\Delta)=0.$
Hence we have
\begin{equation}\label{General-cohomology-thm-eqn4.1.2}
\Ima(V_I^{k-1}(\Delta)\xrightarrow{d_{\pi}}V_I^{k}(\Delta))=0
\end{equation}
for all $I\in S_{k}^\pi(\Delta)\backslash S_{k-1}^\pi(\Delta)$.
\end{enumerate}
By Equation \eqref{General-cohomology-thm-eqn4.1.1} and Equation 
\eqref{General-cohomology-thm-eqn4.1.2},  we have
\begin{equation}\label{General-cohomology-thm-eqn4.1}
\Ima(V_I^k(\Delta)\xrightarrow{d_{\pi}}V_I^{k+1}(\Delta))=0
\end{equation}
for all $I\in S_{k}^\pi(\Delta)$.
By Equation \eqref{General-cohomology-thm-eqn4} and Equation 
\eqref{General-cohomology-thm-eqn4.1}, we have
\begin{equation}\label{General-cohomology-thm-eqn5}
\frac{\ker(V_I^k(\Delta)\xrightarrow{d_{\pi}}V_I^{k+1}(\Delta))}
{\Ima(V_I^{k-1}(\Delta)\xrightarrow{d_{\pi}}V_I^{k}(\Delta))} =V_I^k(\Delta)
\end{equation}
for all $I\in S_{k}^\pi(\Delta)$.
\end{enumerate}

The combination of Equation \eqref{General-cohomology-thm-eqn2}, 
Equation \eqref{General-cohomology-thm-eqn2.1}, 
Equation \eqref{General-cohomology-thm-eqn3} and 
Equation \eqref{General-cohomology-thm-eqn5} proves
\begin{equation}\label{General-cohomology-thm-eqn6}
H_{\pi}^k(X)\cong\bigoplus_{I\in S_k^\pi(\Delta)}V_I^k(\Delta)
\end{equation}
for $1\leq k\leq n$.

\item
For any $0\leq k\leq n$, we have
\begin{equation}\label{General-cohomology-thm-eqn7}
S_k^\pi(\Delta)=\bigcup_{i=0}^k S^\pi(\Delta, i).
\end{equation}
For any $I\in S^\pi(\Delta, i)$, we have $|I_\Delta|=i$. 
By Proposition \ref{weightspace-lem}, we have
\begin{equation}\label{General-cohomology-thm-eqn8}
\dim V_I^k(\Delta)={{n-i}\choose{k-i}}.
\end{equation}

By Equation \eqref{General-cohomology-thm-eqn6}, 
Equation \eqref{General-cohomology-thm-eqn7} and 
Equation \eqref{General-cohomology-thm-eqn8}, we have
\begin{equation*}
\dim H_{\pi}^k(X)=\sum_{i=0}^{k}{{n-i}\choose{k-i}}\# S^\pi(\Delta, i).
\end{equation*}

\item
In the case of $k>n$, 
$H_\pi^k(X)=0$ comes directly from Equation \eqref{General-cohomology-thm-eqn1}.
\end{enumerate}
\end{proof}

\subsection{The proof of Corollary \nameref{Fano-Pcohomology}}
\begin{lemma} \cite{Materov 02}\label{Materov-lem}
Let $X$ be a smooth compact toric variety of dimension $n$, and $L$ be an ample line bundle on $X$. Then we have 
\begin{equation*}
H^i(X, \Omega^j_X\otimes L)=0
\end{equation*}
for all $i>0$ and $0\leq j\leq n$.
\end{lemma}

{\bf Proof of Corollary \nameref{Fano-Pcohomology}}
\begin{proof}
Since $\wedge^j T_X\cong\Omega^{n-j}_X\otimes\wedge^nT_X$, 
by Lemma \ref{Materov-lem}, we have
$$H^i(X,\wedge^j T_X)\cong H^i(X,\Omega^{n-j}_X\otimes\wedge^nT_X)$$ 
for all $0\leq j\leq n$.
If $X$ is a toric Fano manifold,  then $\wedge^n T_X$ is ample. 
By Lemma \ref{Materov-lem}, we have $H^i(X,\wedge^j T_X)=0$ 
for all $i>0$ and $0\leq j\leq n$.
As an application of Theorem \nameref{General-cohomology-thm}, 
we get Corollary \nameref{Fano-Pcohomology}.
\end{proof}

\end{document}